%% file: differential_forms_heat_kernel_for_simply_connected_riemann_surfaces.v3.tex
\documentclass{article}

\include{headers}					
\usepackage[hmargin=1in,vmargin=1in]{geometry}

\title{Heat Kernel for Simply-connected Riemann Surfaces}
\author{Trevor H. Jones \and Dan Kucerovsky}
\date{\today}

\begin{document}
	\maketitle

	\begin{abstract}
		From the uniformization theorem, we know that every Riemann surface has a simply-connected covering space.  Moreover, there are only three simply-connected Riemann surfaces: the sphere, the Euclidean plane, and the hyperbolic plane.  In this paper, we collect the known heat kernels, or Green's functions, for these three surfaces, and we derive the differential forms heat kernels as well.  Then we give a method for using the heat kernels on the universal cover to generate the heat kernel on the underlying surface.
	\end{abstract}

	\section{Introduction} \label{sec:models_h2}
		For the purposes of this paper, we will be concentrating on the three simply-connected Riemann surfaces: the sphere, the plane, and the hyperbolic plane.  We will study the solution operator for the differential forms heat equations on these surfaces.  For a Riemann surface, $M$, the heat equation can be written as,
			\begin{align} \label{eq:he_forms}
				(\di_t + \lp ) \omega(\vx{x},t) & = 0 \nonumber \\
				\omega(\vx{x},0) & =  \nu(\vx{x})
			\end{align}
		where $\omega$ is a differential form depending on a point $\vx{x} \in M$ and on time, $t$, with a bounded, $L^2$ initial condition $\nu(\vx{x})$.  The Laplace-Beltrami operator, $\lp = (d + d^*)^2 = dd^* + d^*d$, is a linear differential operator which maps $k$-forms to $k$-forms.  Thus we do not need to consider mixed degree forms.  This means the heat equation on $M$ is actually three separate equations, one for $0$-forms, which are identified with functions on $M$, one for $2$-forms, which is isomorphic to the $0$-form case, and one for $1$-forms, which yields two differential equations, which are sometimes coupled.

	\section{The Euclidean Plane}
		Let us begin with the heat kernel for differential forms in the Euclidean plane, with the usual metric $\dx{s}{2} = \dx{x}{2} + \dx{y}{2}$.  Let $\omega = A(\vx{x},t) + B(\vx{x},t) \dx{x}{} + C(\vx{x},t) \dx{y}{} + D(\vx{x},t) \dx{x}{} \wedge \dx{y}{},$ be an arbitrary differential form with coefficients having appropriate partial derivatives.  Then the heat equation applied to $\omega$ has the following form:
		\begin{align} \label{eq:he_expl_r2}
			\left( \lp + \di_t \right) \omega =& \left( A_t - A_{xx} - A_{yy} \right)
			+ \left( B_t - B_{xx} - B_{yy} \right) \dx{x}{} \\
			&+ \left( C_t - C_{xx} - C_{yy} \right) \dx{y}{}
			+ \left( D_t - D_{xx} - D_{yy} \right) \dx{x}{} \wedge \dx{y}{}
			= 0 \nonumber
		\end{align}
		The only way for this to be true is for each component to be zero.  This means each of the coefficient functions of $\omega$ must satisfy $u_t = u_{xx} + u_{yy}$.  This has a well known Green's function, or heat kernel,
			$$K(\vx{x}_1, \vx{x}_2, t) = \frac{1}{4\pi t} \exp \left( -\frac{\abs{\vx{x}_1 - \vx{x}_2}^2}{4t} \right),$$
		where $\vx{x}_i = (x_i,y_i)$ and $\abs{\vx{x}}$ is the usual norm on $\bb{R}^2$.  So the heat kernel for differential forms on the Cartesian plane is
		\begin{align} \label{eq:hk_r2}
			K(\vx{x}_1, \vx{x}_2, t) &= \frac{1}{4\pi t} \exp \left( -\frac{\abs{\vx{x}_1 - \vx{x}_2}^2}{4t} \right) \left[ 1_{\vx{x}_1} \tensor 1_{\vx{x}_2} + \dx{x_1}{} \tensor \dx{x_2}{} + \dx{y_1}{} \tensor \dx{y_2}{} + \dx{x_1}{} \wedge \dx{y_1}{} \tensor \dx{x_2}{} \wedge \dx{y_2}{} \right].
		\end{align}

		If we change to polar coordinates, $\vx{x}_i = (r_i, \theta_i)$ with the metric $\dx{s}{2} = \dx{r}{2} + r^2 \dx{\theta}{2}$, the heat kernel becomes
		\begin{align} \label{eq:hk_r2_polar}
			K(\vx{x}_1, \vx{x}_2, t) &= \frac{1}{4\pi t} \exp \left( -\frac{\abs{\vx{x}_1 - \vx{x}_2}^2}{4t} \right) \left[ 1_{\vx{x}_1} \tensor 1_{\vx{x}_2} + \Omega + r_1\dx{r_1}{} \wedge \dx{\theta_1}{} \tensor r_2 \dx{r_2}{} \wedge \dx{\theta_2}{} \right]
		\end{align}
		where $\abs{\vx{x}_1 - \vx{x}_2}^2 = r_1^2 + r_2^2 - 2r_1 r_2 \cos (\theta_1 - \theta_2)$ and
		\begin{align*}
			\Omega =& \cos (\theta_1 - \theta_2) \left[ \dx{r_1}{} \tensor \dx{r_2}{} + r_1 \dx{\theta_1}{} \tensor r_2 \dx{\theta_2}{} \right] + \sin (\theta_1 - \theta_2) \left[ \dx{r_1}{} \tensor r_2 \dx{\theta_2}{} - r_1 \dx{\theta_1}{} \tensor \dx{r_2}{} \right].
		\end{align*}

	\section{The Hyperbolic Plane} \label{sec:h2}
			For the case of the hyperbolic plane, we will use the hyperboloid model, which we denote $H^2$, with the metric $\dx{s}{2} = \dx{r}{2} + \sinh^2 r \dx{\theta}{2}$.  If we let $\omega = A(\vx{x},t) + B(\vx{x},t) \dx{r}{} + C(\vx{x},t) \sinh r \dx{\theta}{} + D(\vx{x},t) \sinh r \dx{r}{} \wedge \dx{\theta}{},$ be an arbitrary differential form with coefficients having appropriate partial derivatives.  Then the heat equation applied to $\omega$ has the following form:
			\begin{align} \label{eq:he_h2}
				(\di_t + \lp) \omega =& \left( A_t - A_{rr} - \frac{\cosh r}{\sinh r} A_r - \frac{1}{\sinh^2 r} A_{\theta \theta} \right) \nonumber \\
				&+ \left( B_t - B_{rr} - \frac{\cosh r}{\sinh r} B_r  + \frac{1}{\sinh^2 r} B - \frac{1}{\sinh^2 r} B_{\theta \theta} + 2 \frac{\cosh r}{\sinh r} C_\theta \right) \dx{r}{} \nonumber \\
				&+ \left( C_t - C_{r r} - \frac{\cosh r}{\sinh r} C_r  + \frac{1}{\sinh^2 r} C - \frac{1}{\sinh^2 r} C_{\theta \theta} - 2 \frac{\cosh r}{\sinh r} B_\theta \right) \sinh r \dx{\theta}{} \nonumber \\
				&+ \left( D_t - D_{r r} - \frac{\cosh r}{\sinh r} D_r - \frac{1}{\sinh^2 r} D_{\theta \theta} \right) \sinh r\dx{r}{} \wedge \dx{\theta}{} = 0.
			\end{align}
			We chose $\sinh r \dx{\theta}{}$ and $\sinh r \dx{r}{} \wedge \dx{\theta}{}$ since these differential forms make the different parts of the heat equation similar, thus simplifying our problem.  It is also worth noting that these forms have unit length under the local inner product.

			Notice that the $\dx{r}{}$ and $\sinh r \dx{\theta}{}$ portion of the equation are coupled with a derivative in $\theta$.  If we consider the Laplacian applied only to a radially symmetric argument, then all terms involving a derivative in $\theta$ disappear, so the equations are no longer coupled.  For forms independent of $\theta$, the heat equation is
			\begin{align} \label{eq:he_h2_rad}
				(\di_t + \lp) \omega =& \left( A_t - A_{r r} - \frac{\cosh r}{\sinh r} A_r \right) + \left(B_t - B_{r r} - \frac{\cosh r}{\sinh r} B_r  + \frac{1}{\sinh^2 r} B \right) \dx{r}{} \nonumber \\
				&+ \left( C_t - C_{r r} - \frac{\cosh r}{\sinh r} C_r  + \frac{1}{\sinh^2 r} C  \right) \sinh r \dx{\theta}{} + \left( D_t - D_{r r} - \frac{\cosh r}{\sinh r} D_r \right) \sinh r \dx{r}{} \wedge \dx{\theta}{} \nonumber \\
				=& 0.
			\end{align}

			The heat kernel for functions, or $0$-forms, on $H^2$ is given by McKean \cite{aubttsoloamonc} and a derivation of the formula can be found in Chavel \cite{eirg}.  The heat kernel for functions on $H^2$ is
				\begin{align} \label{eq:hk0_h2}
					K_0(\vx{x},\vx{y},t) &= \frac{1}{2 \pi} \int_0^\infty P_{-\frac{1}{2} + i\rho}(\cosh d_{H^2}(\vx{x},\vx{y})) \rho \exp\left(-\left(\frac{1}{4} + \rho^2\right)t \right) \tanh \pi \rho d \rho.
				\end{align}
			After some simplification this agrees with the heat kernel given in \cite{aubttsoloamonc}:
				$$K_0(d_{H^2}(\vx{x},\vx{y}), t) = \sqrt{2} \exp\left(-\frac{t}{4}\right)(4\pi t)^{-\frac{3}{2}} \int_{d_{H^2}(\vx{x},\vx{y})}^\infty \frac{s \exp\left(-\frac{s^2}{4t}\right)}{\left( \cosh s - \cosh d_{H^2}(\vx{x}, \vx{y})\right)^{\frac{1}{2}}} ds.$$

			The approach we will use to arrive at the 1-form heat kernel is similar to the approach used by Chavel \cite{eirg}, that is, to find an integral transform for the radial heat equation, and then use translation to arrive at the full heat kernel.  To find the integral transform required for this, we will solve the following eigenfunction problem:
				\begin{align} \label{eq:eigen_h2}
					-E_{rr} - \frac{\cosh r}{\sinh r} E_r + \frac{1}{\sinh^2 r} E &= \lambda E = \left( \frac{1}{4} + \rho^2 \right) E
				\end{align}
			Since the Laplacian is a positive operator, there are no negative or complex eigenvalues.  Since the continuous spectrum of the Laplacian on $H^2$ is $[\frac{1}{4},\infty)$, as stated in \cite{lhfocm,tdfsohs}, we are justified in representing $\lambda$ by $\frac{1}{4} + \rho^2$.  By changing variables, $x = \cosh r$  and $\frac{1}{4} + \rho^2 = - v(v+1)$, \eqref{eq:eigen_h2} can be transformed into Legendre's equation of order $v$ and degree 1.  Since we will be considering only bounded solutions, the solution to \eqref{eq:eigen_h2} is $E_\rho(r) = P^1_{-\frac{1}{2} + i \rho}(\cosh r)$, with $\rho \geq 0$.  We are able to disregard $\rho <0$ since $P^1_{\frac{1}{2} - i\rho}(x) = P^1_{-\frac{1}{2} + i\rho}(x)$.

			With these eigenfunctions, we use the integral transform given by
				\begin{align*} 
					\oh{f}(\rho) &:= \ip{E_\rho(r_y) \dx{r_y}{}}{f(r_y) \dx{r_y}{}} = \int_{H^2} E_\rho(r_y) \dx{r_y}{} \wedge * f(r_y) \dx{r_y}{}.
				\end{align*}
			Since $f$ and $E_\rho$ are independent of $\theta_y$, the integral transform simplifies to
				\begin{align} \label{eq:mehler_fock}
					\oh{f}(\rho) &= 2\pi \int_0^\infty E_\rho(r_y)f(r_y) \sinh r_y \dx{r_y}{}.
				\end{align}
			This is proportional to the Mehler-Fock transform of order $-1$, \cite{tolmagmt}, with inverse transform
				\begin{align} \label{eq:mf_inverse}
					f(r_y) \dx{r_y}{}  & = -\frac{1}{2 \pi} \int_0^\infty \oh{f}(\rho) P^{-1}_{-\frac{1}{2} + i \rho}(\cosh r_y) \rho \tanh \pi \rho \dx{\rho}{} \tensor \dx{r_y}{} \\
					& = \frac{1}{2 \pi} \int_0^\infty \oh{f}(\rho) \frac{\rho \tanh \pi \rho}{\frac{1}{4} + \rho^2} E_\rho(r_y) \dx{\rho}{} \tensor \dx{r_y}{} \nonumber
				\end{align}

			Because the Laplacian is self-adjoint with respect to this inner product, when we apply the transform to the $\dx{r}{}$ portion of the heat equation \eqref{eq:he_h2_rad}, we have
				\begin{align*}
					\ip{E_\rho(r_y) \dx{r_y}{}}{(\di_t + \lp) B(r_y,t) \dx{r_y}{}} & = \oh{B}_t(\rho,t) + \left( \frac{1}{4} + \rho^2 \right) \oh{B}(\rho,t) = 0 \\
					\oh{B}(\rho,0) & = \oh{f}(\rho)
				\end{align*}
			where $B(r_y,0) = f(r_y)$.  Thus $\ds \oh{B}(\rho, t) = \oh{f}(\rho) \exp\left( {- \left(\frac{1}{4} + \rho^2 \right) t} \right).$  So, to solve for $B$ we apply the inverse transform \eqref{eq:mf_inverse}, and we get
				\begin{align*}
					B(r_x,t) \dx{r_x}{} & = \frac{1}{2\pi} \int_0^\infty \left( 2\pi \int_0^\infty E_\rho(r_y) f(r_y) \sinh r_y \dx{r_y}{} \right) \exp\left( {- \left(\frac{1}{4} + \rho^2 \right) t} \right) \frac{\rho \tanh \pi \rho}{\frac{1}{4} + \rho^2} E_\rho(r_x) \dx{\rho}{} \tensor \dx{r_x}{} \\
					& =  \int_0^\infty f(r_y) \sinh r_y \left( \int_0^\infty \exp\left( {- \left(\frac{1}{4} + \rho^2 \right) t} \right) \frac{\rho \tanh \pi \rho}{\frac{1}{4} + \rho^2} E_\rho(r_y) E_\rho(r_x) \dx{\rho}{} \right) \dx{r_y}{} \tensor \dx{r_x}{}
				\end{align*}

			Since $\ds P^1_{-\frac{1}{2} + i\rho}(\cosh r) = \frac{d}{dr} P_{-\frac{1}{2} + i\rho}(\cosh r)$ and
				\begin{align} \label{eq:int_P0}
					P_{-\frac{1}{2} + i\rho}(\cosh r_y)P_{-\frac{1}{2} + i\rho}(\cosh r_x) & = \frac{1}{2\pi} \int_0^{2\pi} P_{-\frac{1}{2} + i\rho}(\cosh r_y \cosh r_x - \sinh r_y \sinh r_x \cos \alpha) \dx{\alpha}{}
				\end{align}
			(see formulas in \cite{homf, tolmagmt, toisap}), we have
				\begin{align*}
					B(r_x,t) \dx{r_x}{} & = \frac{1}{2\pi} \int_0^\infty f(r_y) \sinh r_y \int_0^\infty \exp\left( {- \left(\frac{1}{4} + \rho^2 \right) t} \right) \frac{\rho \tanh \pi \rho}{\frac{1}{4} + \rho^2} \\
					& \hspace*{0.5cm} \times  \frac{\di^2}{\di r_y \di r_x} \int_0^{2\pi} P_{-\frac{1}{2} + i\rho}(\cosh r \cosh r_x - \sinh r_y \sinh r_x \cos \alpha) \dx{\alpha}{} \dx{\rho}{} \dx{r_y}{} \tensor \dx{r_x}{}.
				\end{align*}
			In \eqref{eq:int_P0}, the variable $\alpha$ can be replaced by $\theta_y - \theta_x$, with the integration with respect to $\theta_y$, and keeping the same bounds.  This means the argument of the Legendre function is $\cosh d_{H^2}(\vx{x}, \vx{y})$, where $\vx{x} = (r_x, \theta_x)$ and $\vx{y} = (r_y, \theta_y)$.  Since the integral is over a full period, the result in independent of $\theta_y$ as well as $\theta_x$.  This means that we can replace the partial derivatives with exterior derivatives.  In other words,
				\begin{align} \label{eq:h2_dr}
					B(r_x,t) \dx{r_x}{} & = \frac{1}{2\pi} \int_0^\infty f(r_y) \sinh r_y \int_0^\infty \exp\left( {- \left(\frac{1}{4} + \rho^2 \right) t} \right) \frac{\rho \tanh \pi \rho}{\frac{1}{4} + \rho^2} \nonumber \\
					& \hspace*{0.5cm} \times \int_0^{2\pi} d_\vx{x} d_\vx{y} P_{-\frac{1}{2} + i\rho}(\cosh d_{H^2} (\vx{x}, \vx{y})) \dx{\theta_y}{} \dx{\rho}{}.
				\end{align}

			Using similar arguments, we can write
				\begin{align} \label{eq:h2_dtheta}
					C(r_x,t) \sinh r_x \dx{\theta_x}{} & = \frac{1}{2\pi} \int_0^\infty g(r_y) \int_0^\infty \exp\left( {- \left(\frac{1}{4} + \rho^2 \right) t} \right) \frac{\rho \tanh \pi \rho}{\frac{1}{4} + \rho^2} \nonumber \\
					& \hspace*{0.5cm} \times \int_0^{2\pi} *_\vx{x} *_\vx{y} d_\vx{x} d_\vx{y} P_{-\frac{1}{2} + i\rho}(\cosh d_{H^2} (\vx{x}, \vx{y})) \dx{\rho}{} \dx{r_y}{}.
				\end{align}
			Note that for equations \eqref{eq:h2_dr} and \eqref{eq:h2_dtheta}, that the ``missing'' differentials $\dx{r_y}{}$ and $\dx{\theta_y}{}$ come from the exterior derivative $d_\vx{y}$.

			Now let
				\begin{align*}
					G &= \frac{1}{2\pi} d_\vx{x} d_\vx{y} \int_0^\infty \exp\left( {- \left(\frac{1}{4} + \rho^2 \right) t} \right) \frac{\rho \tanh \pi \rho}{\frac{1}{4} + \rho^2} P_{-\frac{1}{2} + i\rho}(\cosh d_{H^2} (\vx{x}, \vx{y})) \dx{\rho}{}.
				\end{align*}
			Then $\ds B(r_x,t) \dx{r_x}{} = \int_0^{2\pi} \int_0^\infty G \wedge *_\vx{y} f(r_y) \dx{r_y}{}$ and $\ds C(r_x,t) \sinh r_x \dx{\theta_x}{} = \int_0^{2\pi} \int_0^\infty *_\vx{x} *_\vx{y} G \wedge *_\vx{y} g(r_y) \sinh r_y \dx{\theta_y}{}$.  Since $G \wedge *_\vx{y} g(r_y) \sinh r_y \dx{\theta_y}{} = 0$ and $*_\vx{x} *_\vx{y} G \wedge *_\vx{y} f(r_y) \dx{r_y}{}=0$, we can combine this statement into
				\begin{align*}
					B(r_x,t) \dx{r_x}{} + C(r_x,t) \sinh r_x \dx{\theta_x}{} &= \int_0^{2\pi} \int_0^\infty (I + *_\vx{x} *_\vx{y}) G \wedge *_\vx{y} (f(r_y) \dx{r_y}{} + g(r_y) \sinh r_y \dx{\theta_y}{}).
				\end{align*}
			Thus it appears that $(I + *_\vx{x} *_\vx{y}) G$ plays the role of the heat kernel for radially symmetric initial conditions.  We will show that $(I + *_\vx{x} *_\vx{y}) G$ is the heat kernel for 1-forms on the hyperbolic plane.

			\begin{thm} \label{thm:hk1_h2}
				Consider the heat equation on the hyperbolic plane with metric $\dx{s}{2} = \dx{r}{2} + \sinh^2 r \dx{\theta}{}$:
					\begin{align} \label{eq:he_thm}
						(\lp + \di_t) \omega(\vx{x},t) &= 0 \nonumber \\
						\omega(\vx{x},0) & = \nu(\vx{x})
					\end{align}
				where $\omega$ and $\nu$ have the form $f \dx{r}{} + g \sinh r \dx{\theta}{}$ and $\nu$ is a bounded $L^2$ 1-form.  The heat kernel, $K_1(\vx{x},\vx{y},t)$ is given by the formula $K_1(\vx{x},\vx{y},t) = (I + *_\vx{x} *_\vx{y}) G$ where
					\begin{align*}
						G &= \frac{1}{2\pi} d_\vx{x} d_\vx{y} \int_0^\infty \exp\left( {- \left(\frac{1}{4} + \rho^2 \right) t} \right) \frac{\rho \tanh \pi \rho}{\frac{1}{4} + \rho^2} P_{-\frac{1}{2} + i\rho}(\cosh d_{H^2} (\vx{x}, \vx{y})) \dx{\rho}{}.
					\end{align*}
			\end{thm}

			\begin{proof}
				By existence and uniqueness of heat kernels, \cite{thkfpomobg}, \eqref{eq:he_thm} has a unique solution, $\omega$.  We will show that $\ds \omega = \int_{H^2} (I+*_\vx{x} *_\vx{y})G \wedge *_\vx{y} \nu(\vx{y})$.  First, let us define an inner product on forms as $\ds \ip{f}{g} = \int_{H^2} f \wedge * g.$  It is known that $d$ and $d^*$ are adjoint under this inner product and that $\ip{*f}{g} = -\ip{f}{*g}$.  With this notation, we will show that $\omega=\ip{(I+*_\vx{x} *_\vx{y})G}{\nu(\vx{y})}.$  Recall the 0-form heat kernel from \eqref{eq:hk0_h2}, and note that $\ds \di_t G = -d_\vx{x} d_\vx{y} K_0(\vx{x}, \vx{y}, t)$.  We will let $\ds G = \int_t^\infty d_\vx{x} d_\vx{y} K_0(\vx{x}, \vx{y}, \tau) \dx{\tau}{}.$
				\begin{align*}
					\ip{(I+*_\vx{x} *_\vx{y})G}{\nu(\vx{y})} &= \ip{G}{\nu(\vx{y})} + *_\vx{x} \ip{*_\vx{y}G}{\nu(\vx{y})} \\
					&= \int_t^\infty \left[ d_\vx{x} \ip{d_\vx{y} K_0(\vx{x}, \vx{y}, \tau)}{\nu(\vx{y})} + *_\vx{x} d_\vx{x} \ip{*_\vx{y} d_\vx{y} K_0(\vx{x}, \vx{y}, \tau)}{\nu(\vx{y})} \right] \dx{\tau}{} \\
					&= \int_t^\infty \left[ d_\vx{x} \ip{K_0(\vx{x}, \vx{y}, \tau)}{d^*_\vx{y} \nu(\vx{y})} - *_\vx{x} d_\vx{x} \ip{K_0(\vx{x}, \vx{y}, \tau)}{*_\vx{y} d^*_\vx{y} \nu(\vx{y})} \right] \dx{\tau}{} \\
					&= \int_t^\infty \left[ d_\vx{x} d^*_\vx{x} \omega(\vx{x}, \tau) - *_\vx{x} d_\vx{x} *_\vx{x} d^*_\vx{x} \omega(\vx{x}, \tau) \right] \dx{\tau}{} \\
					&= \int_t^\infty \lp \omega(\vx{x}, \tau) \dx{\tau}{} \\
					&= \int_t^\infty -\di_\tau \omega(\vx{x}, \tau) \dx{\tau}{} \\
					&= \omega(\vx{x},t)
				\end{align*}
				In this calculation, we used that fact that the Laplacian commutes with the Hodge star, the exterior derivative, and the coderivative to obtain results like: $0 = d_\vx{x} (\lp + \di_t) \omega = (\lp + \di_t) d_\vx{x} \omega$ with initial conditions $d_\vx{x} \nu(\vx{x}) = d_\vx{x} \omega(\vx{x}, 0)$.  We also used the properties of the 0-form heat kernel to give the results.
			\end{proof}

			Using the notation from the proof of Theorem \ref{thm:hk1_h2}, we can write the 1-form heat kernel on $H^2$ as
				$$K_1(\vx{x}, \vx{y}, t) = (I + *_\vx{x} *_\vx{y}) \int_t^\infty d_\vx{x} d_\vx{y} K_0(\vx{x}, \vx{y}, \tau) \dx{\tau}{}.$$

	\section{The Sphere}
			For the sphere, we use the unit sphere with the metric $\dx{s}{2} = \dx{\phi}{2} + \sin^2 \phi \dx{\theta}{2}$ with
			$\theta \in [0,2\pi)$, $\phi \in [0, \pi]$.  Let $\omega = A(\vx{x},t) + B(\vx{x},t) \dx{\phi}{} + C(\vx{x},t) \sin \phi \dx{\theta}{} + D(\vx{x},t) \sin \phi \dx{\phi}{} \wedge \dx{\theta}{}.$  Here we choose $\sin \phi \dx{\theta}{}$ and $\sin \phi \dx{\phi}{} \wedge \dx{\theta}{}$ for the obvious reasons.

			\begin{align} \label{eq:he_s2}
				(\di_t + \lp)\omega =& \left(A_t - A_{\phi \phi} - \frac{\cos \phi}{\sin \phi} A_\phi - \frac{1}{\sin^2 \phi} A_{\theta \theta} \right) \nonumber \\
				+ & \left( B_t - B_{\phi \phi} - \frac{\cos \phi}{\sin \phi} B_\phi  + \frac{1}{\sin^2 \phi} B - \frac{1}{\sin^2 \phi} B_{\theta \theta} + 2 \frac{\cos \phi}{\sin \phi} C_\theta \right) \dx{\phi}{} \nonumber \\
				+ & \left( C_t - C_{\phi \phi} - \frac{\cos \phi}{\sin \phi} C_\phi  + \frac{1}{\sin^2 \phi} C - \frac{1}{\sin^2 \phi} C_{\theta \theta} - 2 \frac{\cos \phi}{\sin \phi} B_\theta \right) \sin \phi \dx{\theta}{} \nonumber \\
				+ & \left( D_t - D_{\phi \phi} - \frac{\cos \phi}{\sin \phi} D_\phi - \frac{1}{\sin^2 \phi} D_{\theta \theta} \right) \sin \phi \dx{\phi}{} \wedge \dx{\theta}{} = 0.
			\end{align}

			If we have a radially symmetric solution, one that is independent of $\theta$, then the heat equation simplifies as before and we have

			\begin{align} \label{eq:he_s2_rad}
				(\di_t + \lp) \omega =& \left(A_t - A_{\phi \phi} - \frac{\cos \phi}{\sin \phi} A_\phi \right) + \left( B_t - B_{\phi \phi} - \frac{\cos \phi}{\sin \phi} B_\phi + \frac{1}{\sin^2 \phi} B \right) \dx{\phi}{} \nonumber \\
				+ & \left( C_t - C_{\phi \phi} - \frac{\cos \phi}{\sin \phi} C_\phi + \frac{1}{\sin^2 \phi} C \right) \sin \phi \dx{\theta}{} + \left( D_t - D_{\phi \phi} - \frac{\cos \phi}{\sin \phi} D_\phi \right) \sin \phi \dx{\phi}{} \wedge \dx{\theta}{} = 0.
			\end{align}

			Notice a similarity between equations \eqref{eq:he_h2_rad} and \eqref{eq:he_s2_rad}.  By similar techniques, the eigenfunctions for the $0$-form case are $E_n(\phi) = c_n P_n(\cos \phi)$, which has eigenvalues $n^2+n$.  The normalizing coefficient is $\ds c_n~=~\sqrt{\frac{2n+1}{4\pi}}$, using the inner product $\ds \ip{f}{g} = \int_{S^2} f \wedge *g$.  We require the eigenfunction to be bounded and $L^2$, so $n$ is restricted to the positive integers.

			We can write the $0$-form radially symmetric heat kernel on $S^2$ as
				$$G_0(\phi_x,\phi_y,t) = \frac{1}{4\pi} \sum_{n=0}^\infty (2n+1) \exp\left(-n(n+1)t\right) P_n(\cos \phi_x) P_n(\cos \phi_y).$$
			We can simplify this expression by setting $\phi_y=0$.  This means $P_n(\cos \phi_y) = 1$.  We also note that $\phi_x$ is just the distance from the ``north'' pole to $\vx{x}$, so we can recover the full heat kernel on $S^2$ by replacing $\phi_x$ by the distance between the points $\vx{x}$ and $\vx{y}$.  Thus
			\begin{align} \label{eq:hk_s2_0}
				K_0(\vx{x},\vx{y},t) = \frac{1}{4\pi} \sum_{n=0}^\infty \left( 2n+1 \right) \exp\left(-n(n+1)t\right) P_n \left( \cos d_{S^2}(\vx{x},\vx{y}) \right).
			\end{align}

			The situation is similar for the $1$-form case: $E_n(\phi) = c_n P^1_n(\cos \phi)$, with $\ds c_n = \sqrt{\frac{2n+1}{4\pi n (n+1)}}$, and eigenvalue $n^2+n$.  Note that $n > 0$, so there are no harmonic $L^2$ $1$-forms on the unit sphere.  For the $\dx{\phi_x}{} \tensor \dx{\phi_y}{}$ portion of the heat kernel we proceed in a manner similar to that of Section \ref{sec:h2}.  Let
				$$G_1(\phi_x, \phi_y, t) = \frac{1}{4\pi} \sum_{n=1}^\infty \frac{2n+1}{n(n+1)} \exp\left(-n(n+1)t\right) P^1_n(\cos \phi_x) P^1_n(\cos \phi_y) \dx{\phi_x}{} \tensor \dx{\phi_y}{}.$$
			Using some identities of Legendre functions \cite{toisap}, we have $\ds P^1_n(\cos \phi) = \frac{d}{d\phi} P_n(\cos \phi).$  This means we can replace $P^1_n(\cos \phi_x) \dx{\phi_x}{}$ with $\ds d_\vx{x} P_n(\cos \phi_x)$, with a similar expression for the $\vx{y}$ term as well.  So
				$$G_{1\phi}(\phi_x, \phi_y, t) = \sum_{n=1}^\infty F(n,t) d_\vx{x} d_\vx{y} P_n(\cos \phi_x) P_n(\cos \phi_y),$$
			where $F(n,t) = \ds \frac{1}{4\pi} \frac{2n+1}{n(n+1)} \exp\left(-n(n+1)t\right)$.  By similar arguments we find that
				$$G_{1\theta}(\phi_x, \phi_y, t) = \sum_{n=1}^\infty F(n,t) *_\vx{x} *_\vx{y} d_\vx{x} d_\vx{y} P_n(\cos \phi_x) P_n(\cos \phi_y).$$
			Therefore the radially symmetric $1$-form heat kernel can be written as
				$$G_{1}(\phi_x, \phi_y, t) = \sum_{n=1}^\infty F(n,t) \left( I + *_\vx{x} *_\vx{y} \right) d_\vx{x} d_\vx{y} P_n(\cos \phi_x) P_n(\cos \phi_y).$$

			To recover the full heat kernel, we will make use of the identity
				$$\int_0^{2\pi} P_n(\cos \phi_x) P_n(\cos \phi_y) \dx{\theta_y}{} = \int_0^{2\pi} P_n(\cos \phi_x \cos \phi_y + \sin \phi_x \sin \phi_y \cos (\theta_x - \theta_y)) \dx{\theta_y}{}$$
			which is derived from the ``addition theorem" (see 8.814, \cite{toisap}).  The distance formula on the unit sphere is given by $\cos d_{S^2}(\vx{x}, \vx{y}) = \cos \phi_x \cos \phi_y + \sin \phi_x \sin \phi_y \cos (\theta_x - \theta_y)$, so we will use this to simplify the expression.  We will consider the solution of the radial heat equation
			\begin{align*}
				u(\phi_x,t) \dx{\phi_x}{} & = \int_{S^2} \sum_{n=1}^\infty F(n,t) d_\vx{x} d_\vx{y} P_n(\cos \phi_x) P_n(\cos \phi_y) \wedge *_\vx{y} f(\phi_y) \dx{\phi_y}{} \\
				&= \int_{S^2} \sum_{n=1}^\infty F(n,t) \left( d_\vx{x} d_\vx{y} P_n(\cos \phi_x) P_n(\cos \phi_y) \right) \wedge f(\phi_y) \sin \phi_y \dx{\theta_y}{} \\
				&= \int_0^\pi \sum_{n=1}^\infty F(n,t) f(\phi_y) \sin \phi_y d_\vx{x} d_\vx{y} \int_0^{2\pi} P_n(\cos \phi_x) P_n(\cos \phi_y) \dx{\theta_y}{} \\
				&= \int_0^\pi \sum_{n=1}^\infty F(n,t) f(\phi_y) \sin \phi_y d_\vx{x} d_\vx{y} \int_0^{2\pi} P_n(\cos d_{S^2}(\vx{x}, \vx{y})) \dx{\theta_y}{} \\
				&= \int_{S^2} \sum_{n=1}^\infty F(n,t) d_\vx{x} d_\vx{y} P_n(\cos d_{S^2}(\vx{x}, \vx{y})) \wedge *_\vx{y} f(\phi_y) \dx{\phi_y}{}
			\end{align*}
			Following the same argument, we have
			\begin{align*}
				u(\phi_x,t) \sin \phi_x \dx{\theta_x}{} &= \int_{S^2} \sum_{n=1}^\infty F(n,t) *_\vx{x} *_\vx{y} d_\vx{x} d_\vx{y} P_n(\cos d_{S^2}(\vx{x}, \vx{y})) \wedge *_\vx{y} g(\phi_y) \sin \phi_y \dx{\theta_y}{}
			\end{align*}
			In these two calculations, $f(\phi_x) \dx{\phi_x}{}$ and $g(\phi_x) \sin \phi_x \dx{\theta_x}{}$ are initial conditions for the heat equation.  From these calculations we have a candidate for the full $1$-form heat kernel:
			\begin{align} \label{eq:hk1_s2}
				K_1(\vx{x}, \vx{y}, t) &= \frac{1}{4\pi} \sum_{n=1}^\infty \frac{2n+1}{n(n+1)} \exp\left(-n(n+1)t\right) \left( I + *_\vx{x} *_\vx{y} \right) d_\vx{x} d_\vx{y} P_n( \cos d_{S^2}(\vx{x}, \vx{y}) )
			\end{align}
			The proof that this is indeed the $1$-form heat kernel follows exactly that of Theorem \ref{thm:hk1_h2}.

	\section{Tiling Method for Arbitrary Riemann Surfaces} \label{sec:tiling}
		As stated earlier, any Riemann surface will have one of the three simply-connected Riemann surfaces as its universal covering space.  In this section we will show how to generate the heat kernel for a Riemann surface based upon its universal cover and its covering group.

		Let $M$ be an arbitrary Riemann surface, $U$ its universal cover, and $G$ its covering group.  Then there is a projection map $\pi : U \rightarrow M$ such that $\pi(g \cdot \vx{x}) = \pi(\vx{x})$ for all $\vx{x} \in U$ and all $g \in G$.  This map induces a metric on $M$ compatible with the metric on $U$, with distance $d_M(\vx{x}, \vx{y}) = \min_{g \in G} d_M(\tl{\vx{x}}, g \cdot \tl{\vx{y}})$, where $\tl{\vx{x}}$ and $\tl{\vx{y}}$ are fixed pre-images of $\vx{x}, \vx{y} \in M$.  Since the Laplacian is defined in terms of the metric, we have $\lp_U \tl{u}(\vx{x}) = \lp_M u(\pi(\vx{x}))$ for $\vx{x}$ in a small open set in $U$, and $\tl{u}$ is the lift of the function $u$ to the universal cover.  With this information we can state the following theorem:

		\begin{thm} \label{thm:tiling}
			Let $M$ be an arbitrary Riemann surface, $U$ its universal cover, and $G$ its covering group.  Let $K_U(\vx{x}, \vx{y}, y)$ be the heat kernel on $U$.  Then the heat kernel on $M$ is given by
			\begin{align*}
				K_M(\vx{x}, \vx{y}, t) = \sum_{g \in G} K_U(\tl{\vx{x}}, g \cdot \tl{\vx{y}}, t)
			\end{align*}
			for fixed pre-images $\tl{\vx{x}}, \tl{\vx{y}}$ under the projection map $\pi$.
		\end{thm}

		To prove this theorem we use the following definition and results which can be found in \cite{thkonrs}.

		\begin{defi} \label{def:g_periodic}
			Let $M$ be a manifold, and $G$ a group which acts on the manifold.  A function, $f$, on the manifold is said to be $G$-periodic if $f(g\cdot\vx{x}) = f(\vx{x})$ for each $g \in G$ and $\vx{x} \in M$.
		\end{defi}

		\begin{prop} \label{prop:projection}
			Let $M$ be a manifold, $\tl{M}$ be a cover of $M$ with covering group $G$, and let $V$ be a vector space.  Suppose we have a $G$-periodic function $f:\tl{M} \rightarrow V$.  Then there is a unique function $\oh{f}:M \rightarrow V$ such that $f = \oh{f} \circ \pi$, where $\pi$ is the covering map from $\tl{M}$ to $M$.
		\end{prop}

		\begin{proof}
			We know that if $\vx{y} = \pi(\vx{x})$, then $\pi^{-1}(\vx{y}) = \{ g \cdot \vx{x} | g \in G \}$.  Let $f$ be $G$-periodic.  We will define $\oh{f}$ as follows: $\oh{f}(\vx{y}) = f(\pi^{-1}(\vx{y}))$.  The inverse image of the point $\vx{y}$ is a set of points in $\tl{M}$.  However, as stated above, each of those points is the image of one point under the action of the group $G$.  Since $f$ is $G$-periodic, the set $\pi^{-1}(\vx{y})$ is mapped to a single point by $f$, making the function $\oh{f}$ well-defined.

			To show that $\oh{f}$ is unique, assume that $f = \oh{f} \circ \pi = \oh{g} \circ \pi$.  Since $\pi$ is surjective, $\pi(\tl{M}) = M$, hence $\oh{f}(\vx{x}) = \oh{g}(\vx{x})$ for all $\vx{x} \in M$.
		\end{proof}

		\begin{thm} \label{thm:he_g_per}
			Let $M$ be a Riemannian manifold and $G$ be a group of isometries acting on $M$ which preserve orientation.  Then the solution of the heat equation with $G$-periodic initial conditions is $G$-periodic.
		\end{thm}

		\begin{proof}
			Recall that the heat kernel is invariant under the action of isometries, that means $K(g\cdot\vx{x}, g\cdot\vx{y}, t) = K(\vx{x}, \vx{y}, t)$ for all isometries $g$.  To show the solution is $G$-periodic, we will use a substitution $\vx{y} = g \cdot \vx{z}$ with $\dx{\vx{y}}{} = \dx{\vx{z}}{}$.
			\begin{eqnarray*}
				\omega(g \cdot \vx{x},t) & = & \int_M K(g \cdot \vx{x}, \vx{y}, t) \wedge *_\vx{y} f(\vx{y}) \dx{\vx{y}}{} \\
				& = & \int_M K(g \cdot \vx{x}, g \cdot \vx{z}, t) \wedge *_\vx{z} f(g \cdot \vx{z}) \dx{\vx{z}}{} \\
				& = & \int_M K(g \cdot \vx{x}, g \cdot \vx{z}, t) \wedge *_\vx{z} f(\vx{z}) \dx{\vx{z}}{} \\
				& = & \int_M K(\vx{x}, \vx{z}, t) \wedge *_\vx{z} f(\vx{z}) \dx{\vx{z}}{} \\
				& = & \omega(\vx{x},t)
			\end{eqnarray*}
		\end{proof}

		The proof of Theorem \ref{thm:tiling} involves lifting the initial conditions on the surface $M$ to the universal cover $U$.  This creates a $G$-periodic initial condition for the heat equation on $U$, which has, by Theorem \ref{thm:he_g_per}, a $G$-periodic solution.  Then by Proposition \ref{prop:projection}, this solution projects to a unique solution of the heat equation on $M$.

	\section{Conclusions}
		We will conclude by summarizing the results.  First, for the simply-connect surfaces, the $1$-form heat kernel takes the form:
		\begin{align*}
			K_1(\vx{x}, \vx{y}, t) &= \left( I + *_\vx{x} *_\vx{y} \right) \int_t^\infty d_\vx{x} d_\vx{y} K_0(\vx{x}, \vx{y}, \tau) \dx{\tau}{}.
		\end{align*}
		A generalization to manifolds of higher dimension can be found in \cite{hkfom}.  Also, the $2$-form heat kernel is given by
		\begin{align*}
			K_2(\vx{x}, \vx{y}, t) &= *_\vx{x} *_\vx{y} K_0(\vx{x}, \vx{y}, t).
		\end{align*}
		This means that the three heat kernels, that is the $0$-, $1$-, and $2$-form heat kernels, can all be written in terms of the $0$-form heat kernel.  Thus to complete the catalog, here again are the $0$-form heat kernels:
		\begin{align*}
			K_0(\vx{x}, \vx{y}, t) &= \frac{1}{4\pi t} \exp\left(-\frac{d_{\bb{R}^2}(\vx{x}, \vx{y})^2}{4t}\right) & {\rm Euclidean\;plane} \\
			K_0(\vx{x}, \vx{y}, t) &= \frac{1}{2 \pi} \int_0^\infty P_{-\frac{1}{2} + i\rho}(\cosh d_{H^2}(\vx{x},\vx{y})) \rho \exp\left(-\left(\frac{1}{4} + \rho^2\right)t\right) \tanh \pi \rho d \rho & {\rm hyperbolic\;plane} \\
			K_0(\vx{x}, \vx{y}, t) &= \frac{1}{4\pi} \sum_{n=0}^\infty \left( 2n+1 \right) \exp\left(-n(n+1)t\right) P_n \left( \cos d_{S^2}(\vx{x},\vx{y}) \right) & {\rm sphere}
		\end{align*}

		Also, from Section \ref{sec:tiling}, if we know the covering group for an arbitrary Riemann surface, $M$, then the heat kernel on $M$ is given by
		\begin{align*}
			K_M(\vx{x}, \vx{y}, t) = \sum_{g \in G} K_U(\tl{\vx{x}}, g \cdot \tl{\vx{y}}, t),
		\end{align*}
		where $U$ is the universal cover of $M$.

	\bibliography{hkfsrs}
	\bibliographystyle{plain}
\end{document}

%% file: headers.tex
\usepackage{amsfonts}
\usepackage{amssymb}
\usepackage{amsmath}
\usepackage{latexsym}
\usepackage{ifthen}
\usepackage[dvips]{color}

\newcommand{\oh}[1]{\widehat{#1}}
\newcommand{\tl}[1]{\widetilde{#1}}

\newcommand{\abs}[1]{\vert {#1} \vert}
\newcommand{\ip}[2]{\left\langle {#1}, {#2} \right\rangle}

\newcommand{\tensor}{\otimes}

\newcommand{\di}{\partial}

\newcommand{\lp}{\Delta}

\newcommand{\bb}[1]{\mathbb {#1}}
\newcommand{\vx}[1]{{\bf {#1}}}
\newcommand{\dx}[2]{d{#1}^{#2}}

\newcommand{\ds}{\displaystyle}

\newtheorem{defi}{Definition}[section]
\newtheorem{prop}[defi]{Proposition}

\newtheorem{thm}[defi]{Theorem}

\newtheorem{ex}[defi]{Example}
\newtheorem{remar}[defi]{Remark}

\newenvironment{proof}{\noindent {\bf Proof:}}{\hfill{$\blacksquare$} \newline}

\makeatletter
\@addtoreset{equation}{section}
\@addtoreset{figure}{section}
\@addtoreset{table}{section}
\makeatother